\documentclass[12pt]{article}%
   
\usepackage{amsthm, amsmath,enumerate}
\usepackage{amsfonts}
\usepackage{amssymb}
\usepackage{fullpage}

\usepackage{color}

\newtheorem{theorem}{Theorem}

\newtheorem{corollary}[theorem]{Corollary}

\newtheorem{lemma}[theorem]{Lemma}

\newtheorem{proposition}[theorem]{Proposition}

\newcommand\ex{\ensuremath{\mathrm{ex}}}

\title{Tur\'{a}n numbers for Berge-hypergraphs and related extremal problems}

\date{}

\author{
	Cory Palmer\thanks{Department of Mathematical Sciences, University of Montana, \texttt{cory.palmer@umontana.edu}. Research supported by University of Montana UGP Grant \#M25460.}
	\and
	Michael Tait\thanks{Department of Mathematical Sciences, Carnegie Mellon University, \texttt{mtait@cmu.edu}. Research is supported by NSF grant DMS-1606350.}
	\and
	Craig Timmons\thanks{Department of Mathematics and Statistics, California State University Sacramento, \texttt{craig.timmons@csus.edu}.
Research supported in part by Simons Foundation Grant \#359419.}
	\and
	Adam Zsolt Wagner\thanks{Department of Mathematics, University of Illinois at Urbana-Champaign, \texttt{zawagne2@illinois.edu}}
	}

\begin{document}

\maketitle

\begin{abstract}
	 Let $F$ be a graph.  We say that a hypergraph $H$ is a
	{\it Berge}-$F$ if there is a 
	bijection $f : E(F) \rightarrow E(H )$ such that 
	$e \subseteq f(e)$ for every $e \in E(F)$. Note that Berge-$F$ actually denotes a class of hypergraphs. The maximum number of edges in an $n$-vertex $r$-graph
	with no subhypergraph isomorphic to any Berge-$F$ is denoted $\ex_r(n,\textrm{Berge-}F)$.
	In this paper we establish new upper and lower bounds on $\ex_r(n,\textrm{Berge-}F)$ for general graphs $F$, 
	and investigate connections between $\ex_r(n,\textrm{Berge-}F)$ and other 
	recently studied extremal functions for graphs and hypergraphs. 
	One case of specific interest will be when $F = K_{s,t}$.  
		Additionally, we prove a counting result for $r$-graphs of girth five that complements the asymptotic 
	formula $\textup{ex}_3 (n , \textrm{Berge-}\{ C_2 , C_3 , C_4 \} ) = \frac{1}{6} n^{3/2} + o( n^{3/2} )$ of Lazebnik and 
	Verstra\"{e}te [{\em Electron.\ J. of Combin}. {\bf 10}, (2003)].

\end{abstract}

\section{Introduction}

Let $F$ be a graph and $H$ be a hypergraph.
The hypergraph $H$ is a \emph{Berge-$F$} if there is a 
bijection $f : E(F) \rightarrow E(H)$ such that 
$e \subseteq f(e)$ for every $e \in E(F)$.  
Here we are following the presentation of Gerbner and Palmer \cite{gp}.
This notion of a Berge-$F$ extends 
Berge cycles and Berge paths, which have been investigated, to all graphs.  
In general, Berge-$F$ is a family of graphs.  Given an integer $r \geq 2$, write  
\[
\ex_r ( n , \textrm{Berge-}F )
\] 
for the maximum number of edges in an $r$-uniform hypergraph ($r$-graph for short) on $n$ vertices
that does not contain a subhypergraph isomoprhic to a member of Berge-$F$.  
In the case that $r =2$, Berge-$F$ consists of a single graph, namely $F$, and 
$\ex_2 (n ,  \textrm{Berge-}F )$ is the same as the 
usual Tur\'{a}n number $\ex(n , F)$.  

By results of Gy\H{o}ri, Katona and Lemons \cite{GyKaLe} and Davoodi, Gy\H{o}ri, Methuku and Tompkins \cite{DaETAL}, 
we get tight bounds on $\ex_r(n, \textup{Berge-}P_\ell)$ where $P_\ell$ is a path of length $\ell$.
When $F$ is a cycle and $r\geq 3$,  Gy\H{o}ri and Lemons \cite{GyLe4} determined
\[
\ex_r(n,\textup{Berge-}C_{2\ell}) = O(n^{1+1/\ell})
\]
where the multiplicative constant depends on $r$ and $\ell$. This upper bound matches the order of magnitude in the graph case as given by the classical Even-Cycle Theorem of Bondy and Simonovits \cite{BoSi}. Unexpectedly, the same upper-bound holds in the odd case, i.e., for $r \geq 3$ it was shown in \cite{GyLe4} that
\[
\ex_r(n,\textup{Berge-}C_{2\ell+1}) = O(n^{1+1/\ell}).
\]
This differs significantly from the graph case where we may have $\lfloor n^2/4 \rfloor$ edges and no odd cycle.

Instead of a class of forbidden subhypergraphs, much effort has been spent on determining the Tur\'an number of individual hypergraphs. One case closely related to the Berge question is the so-called expansion of a graph.
Fix a graph $F$ and let $r \geq 3$ be an integer. The \emph{r-uniform expansion} of $F$ is the $r$-uniform hypergraph $F^+$ obtained from $F$ by enlarging each edge of $F$ with $r-2$ new vertices disjoint from $V(F)$ such that distinct edges of $F$ are enlarged by distinct vertices.
More formally, we replace each 
edge $e \in E(F)$ with an $r$-set $e \cup S_e$ where the sets $S_e$ have $r-2$ vertices and 
$S_e \cap S_f = \emptyset$ whenever $e$ and $f$ are distinct edges of $H$.  

The $r$-graph $F^+$ has the same number of edges as $F$, but has 
$| V(F) | + |E(F)| (r - 2)$
vertices.  The special case when $F$ is a complete graph $K_k$ has been studied by 
Mubayi \cite{mubayi} and Pikhurko \cite{pikhurko}.  A series of papers 
\cite{kmv I, kmv II, kmv III} by Kostochka, Mubayi, and  Verstra\"{e}te consider 
expansions for paths, cycles, trees, as well as other graphs.  
The survey of Mubayi and Verstra\"{e}te \cite{mv survey} discusses these results as well as many others.              
Given an integer $r \geq 3$ and a graph $F$, we write 
\[
\ex_r ( n , F^+ )
\]
for the maximum number of edges in an $n$-vertex $r$-graph that does not 
contain a subhypergraph isomorphic to $F^+$. A representative theorem in \cite{kmv III} is that
\[
\ex_3(n,K_{s,t}^+) = O(n^{3-3/s})
\]
whenever $t \geq s \geq 3$. It is also shown that this bound is sharp when $t > (s-1)!$.

For a fixed graph $F$, both the Berge-$F$ and expansion $F^+$ hypergraph problems are closely related to counting certain subgraphs in (ordinary) graphs with no subgraph isomorphic to $F$. Let $G$ and $F$ be graphs.  Following Alon and Shikhelman \cite{as}, write  
\[
\ex( n , G , F)
\]
for the maximum number of copies of $G$ in an $F$-free graph with $n$ vertices.  
A graph is \emph{$F$-free} if it does not contain a subgraph isomorphic to $F$.  
The function $\ex( n , G , F)$ was studied in the case $(G, F) = (K_3 , C_5)$ by Bollob\'{a}s and Gy\H{o}ri \cite{bg}, 
and when $(G , F) = (K_3 , C_{2\ell+1})$ by Gy\H{o}ri  and Li \cite{gl}.  Later, 
Alon and Shikhelman \cite{as} initiated a general study of $\ex( n , G , F)$. Among others, they proved

\begin{theorem}[Alon, Shikhelman \cite{as}]\label{alon shikhelman}
	If $F$ is a graph with chromatic number $\chi(F) = k > r$, then
	\[
	\ex(n,K_r,F)=(1+o(1))\binom{k-1}{r}\left(\frac{n}{k-1}\right)^r.
	\]	
\end{theorem}

Note that the famous Erd\H os-Stone theorem is the case when $r=2$.

The next proposition demonstrates a connection between the three extremal functions 
that we have defined so far.  

\begin{proposition}\label{simple ineq}
If $H$ is a graph and $r \geq 2$, then 
\[
\ex(n , K_r , F) \leq \ex_r ( n , \textup{Berge-}F ) \leq \ex_r ( n , F^+ ).
\]
\end{proposition}

One of the main questions that we consider in this work is the relationship 
between these functions for different graphs $F$.  
We will see that in some cases, all three are asymptotically equivalent, while in others they exhibit different 
asymptotic behavior.  In light of the Erd\H{o}s-Stone Theorem, it is not too surprising that the 
chromatic number of $F$ plays a crucial role.    
When $\chi (F) > r$ (the so-called nondegenerate case) we have the following known result
which was stated in \cite{mv survey}.
We provide a proof in Section \ref{section nondegenerate} for completeness.  Given two functions 
$f,g : \mathbb{N} \rightarrow \mathbb{R}$, we write $f \sim g$ if $\lim \frac{ f(n) }{ g(n) } = 1$.    

\begin{theorem}\label{non degenerate thm}
Let $k > r \geq 2$ be integers and $F$ be a graph.  If $\chi (F) = k$, then 
\[
\ex(n , K_r , F) \sim \ex_r ( n , \textup{Berge-}F ) \sim \ex_r ( n , F^+ )  \sim \binom{k-1}{r} \left( \frac{n }{k-1} \right)^r.
\]
\end{theorem}

When $\chi (F) \leq r$ (the so-called degenerate case), we have the following.  

\begin{theorem}\label{degenerate thm}
Let $r \geq k \geq 3$ be integers.  If $F$ is a graph with $\chi (F) = k$, then 
\[
\ex_r (n ,F^+) = o(n^r).
\]
\end{theorem}

It is important to mention that our proofs of Theorem \ref{non degenerate thm} and Theorem 
\ref{degenerate thm} rely heavily on a well-known theorem of Erd\H{o}s (see Theorem \ref{erdos kst} in Section \ref{notation}).    

In the case that $\chi (F) \leq r$, the asymptotic equivalence between these three extremal functions  
need not hold.  As an example, let 
us consider $K_{2,t}$.  In \cite{as}, it is shown that for every fixed $t \geq 2$,
\[
\ex( n , K_3 , K_{2,t} ) = \left( \frac{1}{6} + o(1) \right) (t - 1)^{3/2} n^{3/2}
\]
as $n$ tends to infinity.  However, $\ex_3 (n , \textup{Berge-}K_{2,2} ) \geq \left( \frac{1}{ 3 \sqrt{3} } - o(1) \right) n^{3/2}$ (see for instance Theorem 
5 in \cite{gp}).  Therefore, 
\[
\ex( n , K_3 , K_{2,2} ) \nsim \ex_3 (n , \textup{Berge-}K_{2,2} )
\]
The next result implies that $\ex_3 ( n , \textrm{Berge-}K_{2,t} )$ and $\ex(n , K_3 , K_{2,t} )$ have  
the same order of magnitude for all $t \geq 2$.

\begin{theorem}\label{no k2t}
If $r \geq 3$ and $t \geq r - 1$ are integers, then 
\[
\ex_r ( n , \textup{Berge-}K_{2,t} ) \leq \left( \frac{ r - 1}{t} \binom{t}{r - 1} + 2t + 1 \right) \ex(n , K_{2,t}).
\]
\end{theorem}

We note that during the preparation of this manuscript we became aware of a very similar bound on $\ex_r ( n , \textup{Berge-}K_{2,t} )$ given in a preprint of Gerbner, Methuku and Vizer \cite{ge me vi}.  The result of \cite{ge me vi} gives a better constant 
than the one provided by Theorem \ref{no k2t}, and shows that for all $t \geq 7$, 
\[
\ex (n , K_3 , K_{2,t} ) \sim \ex_3 ( n , \textup{Berge-}K_{2,t} ).
\]

On the other hand, by taking all $\binom{n-1}{2}$ triples that contain a fixed element
we get a $3$-graph with $\Omega(n^2)$ edges that contains no $K_{2,t}^+$.  For more 
on the Tur\'{a}n number of Berge-$K_{2,t}$, see \cite{ge me vi, craig}.  

In the case that $3 \leq r \leq s \leq t$, we have the following upper bound which is a consequence of a more general result 
that is proved in Section \ref{upper bounds subsection}.

\begin{theorem}\label{no kst}
For $3 \leq r \leq s \leq t$ and sufficiently large $n$,  
\[
\ex_r ( n , \textup{Berge-}K_{s,t} ) = O ( n^{ r - \frac{r (r - 1) }{2s} } ).
\]
\end{theorem}

As for lower bounds, we use Projective Norm Graphs and a simple probabilistic argument to construct graphs 
with no $K_{s,t}$, but many copies of $K_r$.    

\begin{theorem}\label{4 lower bound}
Let $s \geq 3$ be an integer.  If $q$ is an even power of an odd prime, then
\[
\ex( 2q^s , K_4 , K_{s + 1  , (s - 1)! + 2} ) \geq \left(\frac{1}{4} - o(1) \right)  q^{3s - 4} .
\]
\end{theorem}

By Proposition \ref{simple ineq}, we have a lower bound on $\ex_4 ( 2q^2 , \textup{Berge-}K_{s+1 , (s-1)! + 2 } )$.  
In the case when $s = 3$, this lower bound that is better than the standard construction 
using random graphs.  This is discussed further in Section \ref{lower bounds section}.

Our final result concerns counting $r$-graphs with no Berge-$\mathcal{F}$ where $\mathcal{F}$ is a 
family of graphs.  
Given an $r$-graph $H$, the \emph{girth} of $H$ is the smallest $k$ such 
that $H$ contains a Berge-$C_k$.  When $k=2$, $C_2$ is the graph with two parallel edges and 
$H$ has girth at least 3 if and only if $H$ is linear.  In general, the girth 
of $H$ is at least $g$ if and only if $H$ contains no Berge-$C_k$ for $k \in \{2,3, \dots , g-1 \}$.     
One of the seminal results in this area is the asymptotic formula 
\[
\ex_3 (n , \textup{Berge-} \{ C_2 , C_3 , C_4 \} )  = \left( \frac{1}{6} +o(1) \right) n^{3/2}
\]
of Lazebnik and Verstra\"{e}te \cite{lv}.  This bound implies that there are at least 
\[
2^{ ( 1/6 + o(1) ) n^{3/2} }
\]
$n$-vertex 3-graphs with girth 5.  Our counting result provides an upper bound that 
matches this lower bound, up to a constant in the exponent, and holds for all $r \geq 2$.   

\begin{theorem}\label{counting theorem girth 5 intro}
	Let $r\geq 2$. Then there exists a constant $c_r$ such that the number of $n$-vertex $r$-graphs of girth at least 5 is 
	at most $2^{c_r n^{3/2}}$.
	\end{theorem}
This is a consequence of a more general result that is given in Section \ref{counting section}.
It was recently shown by Ergemlidze, Gy{\H{o}}ri, and Methuku \cite{ErgemETAL} that 
$\ex_3 (n ,  \textup{Berge-} \{ C_2 ,  C_4 \} )  = \left( \frac{1}{6} +o(1) \right) n^{3/2}$.  We leave it as an open 
problem to determine if Theorem \ref{counting theorem girth 5 intro} holds under the weaker assumption 
that the graphs we are counting may have a Berge-$C_3$.

%


The rest of this paper is organized as follows.  Section \ref{notation} gives the notation and some preliminary results that we will 
need.  Section \ref{general ubs} contains the proof of Theorems \ref{non degenerate thm} and \ref{degenerate thm}.  
Section \ref{kst case} focuses on the special case when $F = K_{s,t}$, while Section \ref{counting section} contains 
the proof of Theorem \ref{counting theorem girth 5 intro} and related counting results.



\section{Notation and preliminaries}\label{notation}

In this section we introduce the notation that will be used throughout the paper.  Additionally, we recall some known 
results that will be used in our arguments, and give a proof of Proposition \ref{simple ineq}.   

For a graph $G$ and a vertex $ \in V(G)$, $k_m (G)$ is the number of copies of $K_m$ in $G$ and $\Gamma_G (v)$ is 
the subgraph of $G$ induced by the neighbors of $v$.  For positive integers $r$, $m$, and $x$, 
we write $K^r (x)$ for the complete $r$-partite $r$-graph with $x$ vertices in each part.
The graph $K_m (x)$ is the complete $m$-partite graph with $x$ vertices in each part and we write $K_m$ instead of 
$K_m (1)$.  

In the previous section we defined the expansion $F^+$ of a graph.  
An important special case is when $F = K_k$ for some $k \geq 2$.  By definition,
the $r$-graph $K_k^+$ must contain a set of $k$ vertices, say $\{v_1 , \dots , v_k \}$, such that 
every pair $\{v_i , v_j \}$ is contained in exactly one edge of $K_k^+$.  We call this set the \emph{core} of $K_k^+$.  
As $k \geq 2$, the core is uniquely determined since every vertex not in the core is contained in exactly one edge and every vertex 
in the core is contained in exactly $k-1$ edges.  The $r$-graph $K_k^+$ has $\binom{k}{2}$ edges and 
$k + \binom{k}{2}(r -2)$ vertices.   


Let $H$ be an $r$-graph.  We define $\partial H$ to be the graph consisting of pairs contained in at least one $r$-edge of $H$, i.e.,
\[
\partial H = \{ \{ x , y \} \subset V(H) : \{ x , y \} \subset e ~ \textrm{for some}~ e \in H \}.
\]
Given $\{x,y \} \in \partial H$, let 
\[
d(x,y) = | \{ e \in H : \{ x ,y  \} \subset e \} |.
\]
The $r$-graph $H$ is \emph{$d$-full} if $d(x,y) \geq d$ for all $\{x,y \} \in \partial H$.  
If more than one hypergraph is present, we may write $d_H ( x,y )$ instead of $d (x,y)$ to avoid confusion.  

The first lemma is a very useful tool for Tur\'{a}n problems involving expansions (see \cite{kmv III, mv survey}).  

\begin{lemma}[Full Subgraph Lemma]
For any positive integer $d$, the $r$-graph $H$ has a $d$-full subgraph $H_1$ with 
\[
e(H_1) \geq e(H)  - (d - 1) | \partial H |.
\]
\end{lemma}
\begin{proof}
If $H$ is not $d$-full, choose a pair $\{x,y \} \in \partial H$ for which $d(x,y) < d$.  Remove all edges that 
contain the pair $\{x,y \}$ and let $H'$ be the resulting graph.  If $H'$ is $d$-full, then we are done.  Otherwise,
 we iterate this 
process which can continue for at most $| \partial H |$ steps.  At each iteration, at most $d-1$ edges are removed.  
\end{proof}

The next simple lemma is useful for finding pairs of vertices with bounded codegree in an 
$r$-graph with no Berge-$F$.   
See Lemma 3.2 of \cite{kmv I} for a similar result.  

\begin{lemma}\label{bounded codegree}
Let $r \geq 3$ be an integer and $H$ be an $r$-graph with no Berge-$F$.  If 
$ \partial H$ contains a copy of $F$, then there is a pair of vertices $\{x,y \}$ such that 
\[
d_H( \{x , y \} ) < e(F).
\]
\end{lemma}
\begin{proof}
Suppose $ \partial H$ contains a copy of $F$, say with edges $e_1 , \dots , e_m$ where $m = e(F)$.  
If every pair $e_i = \{ x_i ,y_i \}$ has 
\begin{equation}\label{bounded codegree eq}
d_H( \{ x_i , y_j \} ) \geq e(F),
\end{equation}
then we can choose $e(F)$ distinct edges $e_i' \in H$ for which $\{x_i , y_i \} \subset e_i'$ for 
all $1 \leq i \leq m$.  This gives a Berge-$F$ in $H$ and so 
(\ref{bounded codegree eq}) cannot hold for all $\{x_i , y_j \}$.   
\end{proof}

A consequence of Lemma \ref{bounded codegree} is that if $H$ is an $r$-graph with 
no Berge-$F$ and $H'$ is a $d$-full subgraph of $H$ with $d \geq e(F)$, then $\partial H'$ must 
be $F$-free.  Lemma \ref{bounded codegree} will be used frequently in Section \ref{upper bounds subsection}.  

Lastly, we will need the following result of Erd\H{o}s \cite{erdos kst theorem}.  

\begin{theorem}[Erd\H{o}s \cite{erdos kst theorem}]\label{erdos kst}
Let $r$ and $x$ be positive integers.  There is an $n_0 = n_0 (r,x)$ and a positive constant $\alpha_{r,x}$ such that 
for all $n > n_0$, any $n$-vertex $r$-graph with more than 
$\alpha_{r,x} n^{ r - 1/ x^{r -1 } }$ edges must contain a complete $r$-partite $r$-graph with 
$x$ vertices in each part.  
\end{theorem}

We conclude this section by providing a proof of Proposition \ref{simple ineq}.

\begin{proof}[Proof of Proposition \ref{simple ineq}]
We begin the proof by showing that the first inequality holds.
Let $G$ be an $n$-vertex graph that is $F$-free and has 
$\ex ( n , K_r , F)$ copies of $K_r$.  Let $H$ be the $r$-graph with the same vertex set as $G$, and 
an $r$-set $e$ is an edge in $H$ if and only if the vertices in $e$ form a $K_r$ in $G$.  The number of 
edges in $H$ is $\ex(n , K_r , F)$.  Suppose that 
$H$ has a Berge-$F$.  Any pair of vertices 
$\{ u , v \}$ that are contained in an edge of $H$ are adjacent in $G$.  Therefore, a Berge-$F$ in $H$ gives 
a copy of $F$ in $G$.  Namely, if $f: E(F) \rightarrow E(H)$ is an injection with the property 
that $\{x , y \} \subset f( \{ x , y \} )$ for all $\{x ,y \} \in E(F)$, then these same pairs $\{x,y \}$ for 
which $\{x , y \} \in E(F)$ are edges of a copy of $F$ in $G$.  We conclude that 
$H$ has no Berge-$F$.

The second inequality is trivial since $F^+$ is a particular Berge-$F$ and so any $r$-graph that 
has no Berge-$F$ has no $F^+$.  
\end{proof}


\section{General upper bounds}\label{general ubs}

In this section, we prove an Erd\H{o}s-Stone type result for $r$-graphs with no $F^+$. 
By Proposition~\ref{simple ineq} this gives general upper bounds on $\ex_r(n,\textrm{Berge-}F)$.
We begin with the non-degenerate case, i.e., when $\chi(F) > r$.


\subsection{Non-degenerate case and the proof of Theorem \ref{non degenerate thm}}\label{section nondegenerate}

In this section we prove Theorem~\ref{non degenerate thm}.
As mentioned  in the introduction, this result was 
stated in Mubayi and Verstra\"{e}te's survey on Tur\'{a}n problems for expansions \cite{mv survey}.  Let $F$ be a graph with chromatic number $\chi(F)=k>r$.
 By Theorem~\ref{alon shikhelman} and Proposition~\ref{simple ineq} it is enough to show that $\ex_r(n,F) ~ \sim \binom{k-1}{r} \left( \frac{n }{k-1} \right)^r.$

It was shown by Mubayi \cite{mubayi} (and later improved by Pikhurko \cite{pikhurko}) that 
\[
\ex_r ( n , K_{k}^+ ) ~ \sim \binom{k-1}{r} \left( \frac{n }{k-1} \right)^r.
\]
Therefore, in order to prove Theorem~\ref{non degenerate thm} it remains to prove the following lemma.


\begin{lemma}\label{erdos stone for expansions}
Let $k > r \geq 2$ be integers and $F$ be a graph with $f$ vertices.  If $\chi(F) = k$ and $\epsilon >0$, then for sufficiently large $n$, 
depending on $k$, $r$, $f$, and $\epsilon$, we have  
\[
\ex_r ( n , F^+ ) < \ex_r ( n , K_{k}^+ ) + \epsilon n^r.
\]
\end{lemma}
\begin{proof}
Let $F$ be a graph with $f$ vertices and $\chi(F) = k$ where $ k > r \geq 2$ are integers.  Let $\epsilon > 0$ and  
$G$ be an $n$-vertex $r$-graph with 
\[
e(G) \geq \ex_r ( n , K_{k}^+ ) + \epsilon n^r.
\]
By the Supersaturation Theorem of Erd\H{o}s and Simonovits \cite{es saturation}, there is a positive constant 
$c  = c( \epsilon )$ such that 
$G$ contains at least $c n^m$ copies of $K_{k}^+$ where 
\[
m := k + \binom{k}{2} (r - 2)
\]
is the number of vertices in the $r$-graph $K_{k}^+$.
Let $Z$ be the $m$-graph with the same vertex set as $G$ where $e$ is an edge of $Z$ if and only if 
there is a $K_{k}^+$ in $G$ with vertex set $e$.

Fix a positive integer $x$ large enough so that 
\[
x^{k} \geq \binom{m}{k} \alpha_{k , f } x^{ k - 1 / f^k} ~~ \mbox{and} ~~ x > f^{k}
\]
where $\alpha_{k  ,f}$ is the constant from Theorem \ref{erdos kst}.  
Note that $x$ depends only on $r$, $k$, and $f$.  
For large enough $n$, depending on $c$ and hence $\epsilon$, we have
\[
e(Z)  \geq c n^m > \alpha_{m,x} n^{ m - \frac{1}{ x^{m - 1} } }
\]
so that $Z$ contains a $K^m (x)$, say with parts $P_1, \dots , P_m$.  Therefore, 
for any 
\[
(p_1 , \dots , p_m ) \in P_1 \times \cdots \times P_m,
\] 
there is a $K_{k}^+$ in $G$ whose vertex set is $\{p_1 , \dots , p_m \}$.

A $K_{k}^+$ must contain $k$ vertices that form the core and since 
\[
| P_1 \times \cdots \times P_m | = x^m,
\]
there are at least $x^m / \binom{m}{k}$ copies of $K_{k}^+$ whose vertex sets are the edges of $Z$, and 
whose vertices in the core come from the same set of $k$~$P_i$'s.  Without loss of generality, we 
may assume that we have $x^m / \binom{m}{k}$ copies of $K_{k}^+$ 
whose core vertices come from $k$-tuples in 
\[
P_1 \times \cdots \times P_{k}.
\]  

Let $Y$ be the $k$-partite $k$-graph with vertex set $P_1 \cup \dots \cup P_{k}$ whose edges are 
the $k$-tuples $(p_1 , \dots , p_{k}) \in P_1 \cup \dots \cup P_{k}$ for which there is a $K_{k}^+$ in $G$ 
whose vertices are an edge of $Z$, and whose core is $\{p_1 , \dots , p_{k} \}$.  Given an edge 
$(p_1 , \dots , p_{k} )$ of $Y$, there are at most $x^{m - ( k +1)}$ edges in $Z$ that 
contain $\{p_1 , \dots ,p_{k} \}$ so that 
\[
e(Y) \geq \frac{ x^m / \binom{ m }{ k} }{x^{m -  k} } = \frac{ x^{k}}{ \binom{ m }{ k} }.
\]
We have chosen $x$ large enough so that  
\[
\frac{x^{k}}{ \binom{m}{ k} } \geq \alpha_{k , f} x^{ k - 1 / f^k}
\]
holds.  By Theorem \ref{erdos kst}, $Y$ contains a $K^{k}(f)$, say with parts $R_1 , \dots , R_{k}$ where 
$R_i \subset P_i$ for $1 \leq i \leq k$.  

Let us pause a moment to recapitulate what we have so far.  For every $k$-tuple
\[
(r_1 , \dots , r_{k}) \in R_1 \times \cdots \times R_{k}
\]
and every $(m -  k )$-tuple 
\[
(p_{k+1} , \dots , p_m ) \in P_{k+1} \times \cdots \times P_m, 
\]
there is a $K_{k}^+$ in $G$ with vertex set $\{r_1 , \dots , r_{k} , p_{k+1} , \dots , p_m \}$ whose core 
is $\{r_1 , \dots , r_{k} \}$.  Since $x > f^{k}$ and each $P_i$ has $x$ vertices, we can choose 
$f^{k}$ tuples 
\[
( p_{k+1} , \dots , p_m ) \in P_{k+1} \times \cdots \times P_m 
\]
such that the corresponding 
sets are pairwise disjoint.  We then pair each one of these sets up with a 
$k$-tuple in $R_1 \times \cdots \times R_{k}$ in a 1-to-1 fashion.  Each such pairing forms a $K_{k}^+$ in $G$
and altogether, we have constructed a $K_{k}(f)^+$ in $G$.  That is, we have an expansion 
of the complete $k$-partite Tur\'{a}n graph with $f$ vertices in each part.    
As $F$ is a subgraph of $K_{k}(f)$, 
$F^+$ is a subgraph of $K_{k}(f)^+$ and so $G$ contains a copy of $F^+$.    
\end{proof}


\subsection{The degenerate case and the proof of Theorem \ref{degenerate thm}}

In this section we prove Theorem~\ref{degenerate thm}, i.e., that
 if $F$ is a graph with $\chi (F) \leq r$, then 
\[
\ex_r ( n , F^+ ) = o( n^r ).
\]
As mentioned in the introduction, the proof is based on Theorem \ref{erdos kst}.
It is an immediate corollary of the following.

\begin{theorem}
If $r \geq 3$ is a fixed integer and $F$ is a graph with $\chi (F) \leq r$, then there is a positive constant 
$C$, depending on $r$ and $F$, such that 
\[
\ex_r ( n , F^+)  \leq C n^{ r- 1 /x^{r - 1} } 
\]
where $x = \binom{r}{2} |V(F)|^2 + |V(F)|$.  
\end{theorem}
\begin{proof}
Assume that $|V(F) | = f$ so that $x = \binom{r}{2} f^2 + f$.  
Let $H$ be an $n$-vertex $r$-graph with $e(H) \geq C n^{ r- 1 /x^{r - 1} }$
where $C$ can be taken large as a function of $r$ and $F$.  We will show that $H$ contains 
a subhypergraph isomorphic to $F^+$.  

For large enough $C$, we have $e(H) > \alpha_{r,x} n^{ r - 1 / x^{r - 1} }$.
By Theorem \ref{erdos kst}, $H$ contains a $K^r (x)$.  Here $K^r (x)$ is the complete $r$-partite $r$-graph 
with $x$ vertices in each part.  
Let $W_1 , \dots , W_r$ be the parts of the $K^r (x)$ in $H$.
Partition each $W_i$ into two sets $U_i$ and $D_i$ where $|U_i| = f$ and 
$|D_i| = \binom{r}{2} f^2$.  
We are going to construct a $K_r (f)^+$ in $H$ one edge at a time.  
The vertices that lie in exactly one edge of the $K_r (f)^+$ will 
come from the sets $D_1 \cup \dots \cup  D_r$, 
and the other vertices will come from $U_1 \cup \dots \cup U_r$.  

Let $x \in U_1$ and $y \in U_2$.  Choose exactly one vertex, say $z_i$, from $D_i$ for $3 \leq i \leq r $
and make $\{x , y , z_3 , \dots , z_r \}$ an edge.  
Next we pick a new pair $x' \in U_1$ and $y' \in U_2$ and 
choose exactly one vertex, say $z_i'$, from $D_i \backslash \{ z_i \}$ for $3 \leq i \leq r$.  Make 
$\{x' , y'  , z_3 ' , \dots , z_r ' \}$ an edge.  We can continue this process and in the next round, 
we add an edge $\{x '' , y ''  , z_3 ''  , \dots , z_r ''  \}$ where 
$\{x ''  , y ''  \}$ is a new pair ($x  ''  \in U_1 , y ''  \in U_2$) and the sets 
$\{z_3 , \dots , z_r \}$, $\{z_3 '  , \dots z_r '  \}$, and $\{ z_3 ''  , \dots , z_r ''  \}$ are all pairwise disjoint.  

Since $|D_i | \geq  f^2$, we can continue this process for all pairs of vertices in $U_1$ and $U_2$.
Even more, since $|D_i | \geq \binom{r}{2} f^2$, this process can continue until we have 
considered all pairs $U_i$ and $U_j$ with $1 \leq i < j \leq r$.  
When the process is completed, we have constructed a $K_r (f)^+$ in $H$.  
Now since $F$ is a subgraph of $K_r (f)$, we have that $F^+$ is a subgraph of 
$K_r (f)^+$ and this completes the proof of the theorem.  
\end{proof}


\section{Forbidding Berge-$K_{s,t}$}\label{kst case}

In this section we investigate the special case of forbidding the Berge-$K_{s,t}$. 

\subsection{Upper bounds and the proof of Theorems \ref{no k2t} and \ref{no kst}}\label{upper bounds subsection}

 We begin with an easy lemma.  

\begin{lemma}\label{counting clique lemma}
If $2 \leq m \leq s$, then 
\[
\ex(n , K_{m} , K_{1,s} ) \leq \left( \frac{ n}{s} \right) \binom{s}{m}.
\]
\end{lemma}
\begin{proof}
Let $G$ be an $n$-vertex $K_{1,s}$-free graph.  Every vertex of $G$ has degree at most $s - 1$ so
\[
k_{m} (G) = \frac{1}{m} \sum_{v \in V(G) } k_{m-1}(\Gamma_G (v)) \leq \frac{n}{m} \binom{s-1}{m-1} = \frac{n}{s} \binom{s}{m}.
\]
\end{proof}  

We are now ready to prove Theorem~\ref{no k2t}.

\begin{proof}[Proof of Theorem~\ref{no k2t}]
Fix integers $3 \leq r \leq t$ and let $H$ be an $n$-vertex $r$-graph with no Berge-$K_{2,t}$.  
Let 
\begin{center}
$H_0 = H$, $F_0 = \partial H_0$, 
\end{center}
and $G_0$ be the graph with no edges and vertex set $V( H_0)$.  
If the graph $F_0$ is not $K_{2,t}$-free, then by Lemma \ref{bounded codegree},  
there is a pair of vertices $\{x_1 , y_1 \}$ with 
\[
d_{H_0} ( \{x_1 , y_1 \} ) < 2t.
\]
Now let $H_1$ be obtained from $H_0$ by removing all of the edges that contain $\{ x_1 , y_1 \}$ and 
\[
F_1 = \partial H_1.
\]
Let $G_1$ be the graph obtained by adding the edge $\{x_1 , y_1 \}$ to $G_0$.  

Now we iterate this process.  That is, for $i \geq 1$, we proceed as follows.  

If $F_{i-1}$ is not $K_{2,t}$-free, 
then by Lemma \ref{bounded codegree} there is a pair of vertices $\{x_i ,y_i \}$ in $H_{i - 1}$ with 
\[
d_{H_{i-1} } ( \{ x_i , y_i \} ) < 2t.
\]
Let $H_i$ be the $r$-graph obtained from $H_{i-1}$ by removing all of the edges that contain the pair $\{x_i , y_i \}$, let
\[
F_i = \partial H_i
\]
and $G_i$ be the graph obtained by adding the edge $\{x_i , y_i \}$ to $G_{i-1}$.  Observe that 
\[
e(H_i) > e(H_{i-1}) - 2t.  
\]
Suppose that this can be done for $l : = \delta e(H)$ steps where 
\[
\delta := \frac{1}{ \frac{r-1}{t} \binom{t}{r-1}  + 2t + 1 }.
\]
Consider the graph $G_l$.  This graph 
has $l$ edges and must be $K_{2,t}$-free otherwise, we find a $K_{2,t}$ in $H$ since edges in $G_i$ come from different edges in 
$H$.  Thus,
\[
\delta e(H) = e(G_l)  \leq \ex(n , K_{2,t} )
\]
so
\[
e(H) \leq \frac{1}{ \delta} \ex( n , K_{2,t} )
\]
and we are done.  

Now assume that this procedure terminates for some $l \in \{0,1, \dots , \delta e(H)  \}$
where $ l = 0$ is allowed.    
The graph $F_l$ must be $K_{2,t}$-free so 
\[
| \partial H_l | = e(F_l) \leq \ex(n , K_{2,t} ).
\]
Let 
\[
d_t = \frac{r-1}{t} \binom{t}{r-1} + 1.
\]
The values $d_t$ and $\delta$ satisfy the equation
\[
\frac{d_t}{1 - 2t \delta } = \frac{1}{ \delta}.
\]
If $e(H) \leq \frac{d_t}{1 - 2t \delta } \ex(n , K_{2,t} )$, then we are done.  
For contradiction, suppose that 
\begin{equation}\label{contra}
e(H) > \frac{d_t}{1 - 2t \delta } \ex(n , K_{2,t} ) .
\end{equation}

Let $H'$ be a $d_t$-full subgraph of $H_l$ with 
\begin{eqnarray*}
e(H') & \geq & e(H_l) - d_t | \partial H_l| \geq e(H_0) - 2t l  - d_t \ex(n , K_{2,t} )  \\
& \geq & e(H_0) - 2t  \delta e(H)  - d_t \ex(n , K_{2,t} )  \\
& = & (1 - 2t \delta ) e(H) - d_t \ex(n , K_{2,t} )  > 0 
\end{eqnarray*}
where the last inequality follows from (\ref{contra}).  

Let $F' = \partial H' $.  We now make 
a few observations about the graph $F'$.  First note that $F'$ contains edges since $e(H' ) > 0$.  
Second, $F'$ is $K_{2,t}$-free.  This is because $H'$ is a subgraph of $H_l$ and so $F'$ is a subgraph of $F_l$, but 
$F_l$ is $K_{2,t}$-free.    
Let $v$ be a vertex of $F'$ with positive degree.  The subgraph of $F'$ induced by the neighbors 
of $v$, which we denote by $\Gamma_{F'}(v)$, is $K_{1,t}$-free.  
Since $t \geq r -1$, we have by Lemma \ref{counting clique lemma} that
\begin{equation}\label{k2t eq1}
k_{r - 1} ( \Gamma_{F'} (v) )  \leq \left( \frac{ d_{F'} (v) }{ t} \right) \binom{t}{r - 1} .
\end{equation}
Now we find a lower bound for $k_{r - 1} ( \Gamma_{F' }(v) )$.
Let $w$ be a vertex in $\Gamma_{F'}(v)$.  
Since $H'$ is $d_t$-full, there are at least $d_t$~$r$-sets in $H'$ which contain 
$\{v , w \}$.  
Now if $e$ is an $r$-set in $H'$ that contains $\{v , w \}$, then the $(r-1)$-set $e \backslash \{v \}$ forms a 
$(r-1)$-clique in $\Gamma_{F'} (v)$.  Therefore, this holds for any of the 
$d_{F'}(v)$ vertices in $\Gamma_{F'} (v)$ and so  
\begin{equation}\label{k2t eq2}
k_{r-1} ( \Gamma_{F'} (v) ) \geq \frac{1}{r-1} d_{F'}(v) d_t .
\end{equation}
Combining (\ref{k2t eq1}) and (\ref{k2t eq2}) gives 
\[
\frac{1}{r-1} d_{F'} (v) d_t \leq k_{r-1} ( \Gamma_{F'} (v) ) \leq \left( \frac{ d_{F'} (v) }{ t} \right) \binom{t}{r - 1}.
\] 
As $d_{F'}(v) > 0$, the above inequality implies 
\[
d_t \leq \frac{r-1}{t} \binom{t}{r-1}
\]
which is a contradiction since $d_t = \frac{r-1}{t} \binom{t}{r-1} + 1$.  
We conclude that (\ref{contra}) cannot hold and this completes the proof.   
\end{proof}



We now prove a general upper bound that implies Theorem \ref{no kst}.  A similar result was proved in 
\cite{ge me vi}.  We have chosen to use notation similar to that of \cite{ge me vi} to highlight the correspondence.  

\begin{theorem}\label{new general ub}
Suppose $F$ is a bipartite graph and that there is a vertex $x \in V(F)$ such that for all $m \geq 1$, 
\[
\textup{ex}( m , K_{r-1} , F - x) \leq c m^i
\]
for some positive constant $c$ and integer $i \geq 1$.  If $r \geq 3$ is an integer, $v_F$ is the number 
of vertices of $F$, and $e_F$ is the number of edges of $F$, then for large enough $n$, depending on $r$ and $F$,   
\[
\textup{ex}_r ( n, \textrm{Berge-}F) \leq 4 c (r - 1) 2^{i-1} \frac{ \textup{ex}(n , F )^i }{ n^{i-1} } + 4(v_F  + e_F) n^2.
\]
\end{theorem}
\begin{proof}
Let $F$ be a bipartite graph satisfying the assumptions of the theorem.  
Let $H$ be an $n$-vertex $r$-graph with no Berge-$F$.
If $e(H) \leq 4 ( v_F  + e_F) n^2$, then we are done.  
Assume otherwise and that $\theta$ satisfies
\[
e(H) = 4( v_F + e_F) n^{r - \theta}.
\]
Note that $r - \theta \geq 2$ since $e(H) > 4 (v_F + e_F)n^2$.  
Let $H_1$ be a $(v_F + e_F)$-full subgraph of $H$ with 
\begin{eqnarray*}
e(H_1) & \geq & e(H) - ( v_F + e_F) | \partial H| \geq 4(v_F + e_F) n^{r- \theta} - (v_F + e_F)n^2 \\
& \geq & 3 ( v_F + e_F) n^{r - \theta}.
\end{eqnarray*}
If $\partial H_1$ contains a copy of $F$, then since $H_1$ is $(v_F + e_F)$-full, we have a Berge-$F$ in $H_1$ (and thus $H$) by Lemma 
\ref{bounded codegree}; a contradiction
Thus, $\partial H_1$ is $F$-free and therefore $| \partial H_1 | \leq \textup{ex}( n , F)$.  
Let  
\[
d = \frac{ (  v_F + e_F  ) n^{r - \theta} }{ \textup{ex}(n , F) } .
\]
Let $H_2$ be a $d$-full subgraph of $H_1$ with 
\begin{eqnarray*}
e(H_2) & \geq &  e(H_1) - d | \partial H_1 | \geq 3 ( v_F + e_F ) n^{r - \theta} - d \cdot \textup{ex}( n , F)  \\
& = & 2 (v_F + e_F )  n^{r - \theta}.
\end{eqnarray*}
Let $H_3$ be the subgraph of $H_2$ obtained by removing all isolated vertices and let $G = \partial H_3$.

The graph $G$ is $F$-free as it is a subgraph of $\partial H_1$, so $e(G) \leq \textup{ex}(n , F)$.  Let $v$ be a vertex of $G$ with 
\begin{equation}\label{degree of v}
d_G(v) \leq \frac{2 \textup{ex}(n , F) }{ n}.
\end{equation}
Let $\Gamma_G (v)$ be the subgraph of $G$ induced by the neighbors of $v$ in $G$.  
As $H_3$ is $d$-full, we have that 
there are at least $d$ edges in $H_3$ that contain both $v$ and $w$ for any vertex $w \in \Gamma_G (v)$.  Each such 
edge in $H_3$ gives rise to a $K_{r-1}$ in $\Gamma_G (v)$ that contains $w$.  Therefore, 
\[
k_{r - 1} ( \Gamma_G (v) ) \geq \frac{ d_G (v) d }{r -1}.
\]
However, $G$ is $F$-free and so $\Gamma_G (v)$ is $(F-x)$-free where $x$ is any vertex in $F$.  We conclude that 
\[
\frac{ d_G (v) d }{r - 1} \leq k_{r - 1} ( \Gamma_G (v) ) \leq \textup{ex}( d_G(v) , K_{r-1} , F-x )
\]
for any $x \in V(F)$.  Using our hypothesis and the definition of $d$, this inequality can be rewritten as 
\[
\frac{ d_G(v) ( v_F + e_F ) n^{r - \theta} }{( r-1) \textup{ex}(n,F) }  
\leq 
c d_G(v)^i.
\]
We can cancel a factor of $d_G (v)$ and rearrange the above inequality to get, using (\ref{degree of v}), that
\[
(v_F + e_F) n^{r - \theta} \leq c (r - 1) \textup{ex}(n , F) \left( \frac{ 2 \textup{ex}(n , F) }{n} \right)^{ i - 1}.
\]
Since $e(H) = 4 ( v_F + e_F ) n^{r - \theta}$,  
\[
e(H) \leq 4 c (r - 1) 2^{i-1} \frac{ \textup{ex}(n , F) ^i }{ n^{i-1} }.
\]
\end{proof}

We complete this section by using Theorem \ref{new general ub} to prove Theorem \ref{no kst}.  We must 
show that 
\[
\ex_r ( n , \textup{Berge-}K_{s,t} ) = O ( n^{ r- \frac{ r ( r- 1) }{2s} } )
\]
for $3 \leq r \leq s \leq t$.    

\begin{proof}[Proof of Theorem \ref{no kst}]
Let $3 \leq r \leq s \leq t$ be integers.  
By a result of Alon and Shikhelman (see Lemma 4.2 \cite{as}), 
\[
\textup{ex} ( m , K_{r-1} , K_{s-1 , t} ) \leq \left( \frac{1}{(r-1)! } - o_m(1) \right) ( t - 1)^{ \frac{ (r-1)(r-2) }{2(s-1)} } 
m^{ r -1 - \frac{ (r-1)(r-2) }{2(s-1) } }.
\]
We apply Theorem \ref{new general ub} with $c$ sufficiently large as a function of $r$, $s$, and $t$, with 
\[
i =  r -1 - \frac{ (r-1)(r-2) }{2(s-1) },
\]
and use the well-known bound $\textup{ex}(n , K_{s,t}) = O(n^{2 - 1/s} )$ to get 
that for large enough $n$, 
\[
\textup{ex}_r (  n , \textrm{Berge-}K_{s,t} ) = O ( n^{ (2 - 1/s) i - i + 1} ).
\]
Here the implied constant depends only on $r$, $s$, and $t$.  
A short calculation shows that 
\[
(2 - 1/s) i - i + 1 = r - \frac{ r (r - 1) }{2s}
\]
and this completes the proof. 
\end{proof}

\subsection{Lower Bounds and the proof of Theorem \ref{4 lower bound}}\label{lower bounds section}

By Proposition \ref{simple ineq}, 
\[
\ex(n , K_r , F) \leq \ex_r ( n , \textup{Berge-}F ) \leq \ex_r ( n , F^+ ).
\]
We can use this inequality together with the results of \cite{as} to immediately obtain lower bounds on 
$\ex_r ( n , \textup{Berge-}F)$ and $\ex_r ( n , F^+)$.  

\begin{theorem}[Alon, Shikhelman \cite{as}]
For $r \geq 2$, $s \geq 2r - 2$, and $t \geq (s - 1)! + 1$, 
\[
\left( \frac{1}{r!} + o(1) \right) n^{ r - \frac{ r (r - 1) }{2s} } 
\leq \ex(n , K_r , K_{s,t}) . 
\]
For $s \geq 2$ and $t \geq ( s- 1)! + 1$, 
\[
\left( \frac{1}{6} + o(1) \right) n^{ 3 - \frac{ 3 }{s} } 
\leq \ex(n , K_3 , K_{s,t}) . 
\]
\end{theorem}

Kostochka, Mubayi, and Verstra\"{e}te \cite{kmv III} proved that for any $3 \leq s \leq t$, 
\[
\ex_3 ( n , K_{s,t}^+ )  = O (n^{3 - 3/s} ).
\]
It follows from Proposition \ref{simple ineq} that all three of the functions 
\[
\ex(n , K_3 , K_{s,t}), ~ \ex_3 ( n , \textup{Berge-}K_{s,t} ) , ~ \mbox{and}~\ex_3 ( n , K_{s,t}^+ ) 
\]
are $O(n^{3 - 3/s} )$, and in the case that $t \geq (s - 1)!+1$, they are $\Theta ( n^{3 - 3/s} )$.  



Before giving our lower bounds we introduce some notation.  Let $G$ be a graph and $A$ and $B$ be disjoint subsets of 
$V(G)$.  Write $G[A]$ for the subgraph of $G$ induced by $A$ and $G(A,B)$ for the spanning subgraph of 
$G$ whose edges are those with one endpoint in $A$ and the other in $B$.

\begin{lemma}\label{lemma 1}
Let $3 \leq s \leq t$ be integers.  
Let $G$ be a graph and $V(G) = A \cup B$ be a partition of the vertex set of $G$.  If $G[A]$ is $K_{2,2}$-free, 
$G[B]$ is $K_{2,2}$-free, and $G(A,B)$ is $K_{s,t}$-free, then $G$ is $K_{s+1,t+1}$-free.
\end{lemma}
\begin{proof}
For contradiction, suppose that 
\begin{center}
$\{ x_1  , \dots , x_{s+1} \}$ and $\{ y_1 , \dots , y_{t+1} \}$
\end{center}
are parts of a $K_{s+1 , t+1}$ in $G$.  
Assume first that $A$ contains at least $s$ of the $x_i$'s.  Since $s > 2$ and $G[A]$ is $K_{2,2}$-free, 
$A$ can contain at most one $y_j$ so that $B$ contains at least $t$ of the $y_j$'s.  This, however, gives a $K_{s,t}$ in 
$G(A,B)$ which is a contradiction.  By symmetry, $B$ cannot contain $s$ of the $x_i$'s and so 
we may assume that $A$ contains at least two $x_i$'s and $B$ contains at least two $x_i$'s.  Here we are using the 
fact that $s + 1 \geq 4$.  As $G[A]$ and $G[B]$ are $K_{2,2}$-free, each of $A$ and $B$ can contain at most 
one $y_j$ which is a contradiction since $t + 1 > 2$.    
\end{proof}


Our construction will make use of the Projective Norm Graphs of 
Alon, Koll\'{a}r, R\'{o}nyai, and Szab\'{o} \cite{ars, krs}.
Let $q$ be a power of an odd prime, $s \geq 2$ be an integer, and 
$N: \mathbb{F}_{q^{s-1} } \rightarrow \mathbb{F}_q$ be the norm function defined by 
\[
N(X) = X^{1 + q + q^2 + \dots + q^{s - 2} }.
\]
The Projective Norm Graph, which we denote by $H(s,q)$, is the graph 
with vertex set $\mathbb{F}_{q^{s  - 1} } \times \mathbb{F}_q^*$ where $(x_1 , x_2)$ is 
adjacent to $(y_1 , y_2)$ if $N(x_1 + y_1 ) = x_2 y_2$.  
We will use a bipartite version of this graph.  Let $H^b (s,q)$ be the bipartite graph 
whose parts are $A$ and $B$ where $A$ and $B$ are disjoint copies of 
$\mathbb{F}_{q^{s  - 1} } \times \mathbb{F}_q^*$, and 
$(x_1 , x_2)_A$ in $A$ is adjacent to $(y_1 , y_2)_B$ in $B$ if 
\[
N(x_1 + y_1) = x_2 y_2.
\]
It is shown in \cite{ars} that $H(s,q)$ is $K_{s,(s-1)!+1}$-free. A similar argument gives that  $H^b(s,q)$ is $K_{s,(s-1)!+1}$-free.

\begin{lemma}\label{lemma 2}
Let $s \geq 3$ be a fixed integer.  The graph $H^b(s,q)$ has at least 
\[
(1 - o(1)) \dfrac{ q^{4 ( s - 1) } }{4}
\]
copies of $K_{2,2}$ where $o(1) \rightarrow 0$ as $q \rightarrow \infty$.  
\end{lemma}
\begin{proof}
We will use a known counting argument to obtain a lower bound on the number 
of $K_{2,2}$'s in a $d$-regular bipartite graph with $n$ vertices in each part.  

Suppose that $F$ is a $d$-regular bipartite graph with parts $X$ and $Y$ where $|X| = |Y| = n$.  
Write $X^{(2)}$ for the set of all subsets of size 2 in $X$ and 
write $\hat{d}( \{x , x'  \})$ for the number of vertices that are adjacent to both $x$ and $x'$.  
We have 
\begin{equation}\label{lemma 2 eq1}
\sum_{ \{x , x' \} \in X^{ (2) } }  \hat{d}( \{x , x' \} )  = 
\sum_{y \in Y } \binom{ d(y) }{2} = n \binom{d}{2}.
\end{equation}
The number of $K_{2,2}$'s in $F$ is 
\[
\sum_{ \{ x , x' \} \in X^{ (2) } } \binom{ \hat{d}( \{x , x'  \} ) }{2}  \geq
\binom{n}{2} 
\binom{     \binom{n}{2}^{-1} \sum_{ \{ x , x' \} \in X^{ (2) } } \hat{d}( \{x , x'  \} )  }{2}
\geq \binom{n}{2} 
\binom{   n \binom{d}{2}  / \binom{n}{2} }{2}
\]
where the first inequality is by convexity and the second is by (\ref{lemma 2 eq1}).  
Therefore, the number of $K_{2,2}$'s in $F$ is at least 
\[
\frac{1}{2} n \binom{d}{2} \left( \frac{ n \binom{d}{2} }{ \binom{n}{2} } - 1 \right) 
= 
\frac{nd(d - 1) }{4} \left( \frac{ d ( d - 1) }{ n - 1} - 1 \right).
\]

The graph $H^b(s,q)$ has $q^{s-1} (q - 1)$ vertices in each part and is $(q^{s-1} - 1)$-regular.  
For $s \geq 3$, we have that the number of $K_{2,2}$'s in $H^b (s,q)$ is at least 
\[
(1 - o(1)) \frac{ q^{4s - 4 } }{4}
\]
where $o(1) \rightarrow 0$ as $q \rightarrow \infty$.  
\end{proof} 


Let $q$ be a power of an odd prime and $R_q$ be the graph with vertex 
set $\mathbb{F}_q \times \mathbb{F}_q$ where $(a_1 , a_2)$ is adjacent to $(b_1 , b_2)$ if and only if 
$a_1 + b_1 = a_2 b_2$.  The graph $R_q$ has 
$q^2$ vertices. It is easy to check (see \cite{lw}) that $R_q$ has $\frac{1}{2} q^2 ( q - 1)$ edges and no copy of $K_{2,2}$.  

We now have all of the tools that we need in order to prove Theorem \ref{4 lower bound}.
We must show that for $s \geq 3$ and $q$ an even power of an odd prime, 
\[
\ex( 2q^s , K_4 , K_{s + 1  , (s - 1)! + 2} ) \geq \left(\frac{1}{4} - o(1) \right)  q^{3s - 4} .
\]

\begin{proof}[Proof of Theorem~\ref{4 lower bound}]
Let $A$ and $B$ be disjoint sets of $q^s$ vertices each.  Choose $A' \subset A$ and $B' \subset B$ arbitrarily
with $|A'| = |B'| = q^{s - 1} ( q - 1)$.  Put a copy of $H^b (s,q) $ between $A'$ and $B'$.
Finally, pick two independent random copies of $R_{q^{s/2}}$ on vertex sets $A$ and $B$ and
 let $G$ be the resulting graph.  
Observe that a given pair in $A$ (or $B$) is adjacent with probability $q^{ - s/2}$.  
By Lemma 
\ref{lemma 2} and independence, the expected number of 
copies of $K_4$ in $G$ is at least 
\[
\left(\frac{1}{4} - o(1) \right)  q^{ 4 ( s- 1) }  \left( \frac{1}{ q^{s/2} } \right)^2 = 
\left(\frac{1}{4}  - o(1) \right)  q^{3s - 4} .
\]
Fix a graph $G_q$ with at least this many copies of $K_{4}$.  
Clearly $G_q [A]$ and $G_q [B]$ are both $K_{2,2}$-free and the edges of $G_q (A,B)$ form a $H^b(s,q)$ which is $K_{s , (s - 1)! + 1}$-free.  
By Lemma 
\ref{lemma 1}, $G_q$ is $K_{s + 1 , (s - 1)! + 2}$-free.  
\end{proof}


A density of primes argument, Theorem \ref{4 lower bound}, and Theorem \ref{no kst} give the following result
for 4-graphs.  

\begin{corollary}
If $s \geq 3$ is an integer, then for sufficiently large $n$, there are positive constants $c_s$ and $C_s$ such that 
\[
c_s n^{3 - 4/s} \leq \ex_4 ( n , \textup{Berge-}K_{ s + 1 , (s - 1)! + 2 } ) \leq C_s n^{4 - 6 / (s + 1) }.
\]
\end{corollary} 

In particular, there is a positive constant $c$ such that  
\begin{equation}\label{all 4 case}
c n^{5/3} \leq \ex  ( n , K_4 , K_{4,4} )
\end{equation}
provided $n$ is sufficiently large.  This lower bound is better than what one 
obtains using a simple expected value argument and random graphs. 
Indeed, suppose $G$ is a random 
$n$-vertex graph where a pair forms an edge with probability $p$, independently of the other edges. 
Let $X$ be the number of 4-cliques in $G$ and $Y$ be the number of $K_{4,4}$'s in $G$.  We have 
\[
\mathbb{E} (X - Y ) \geq \left( \frac{n}{4} \right)^4 p^6 - n^8 p^{16}.
\]
If $p = \left( \frac{3}{2^{11}} \right)^{1/10} n^{ - 2/5}$, then 
\[
\mathbb{E}(X-Y) \geq 0.00004 n^{8/5}.
\]
This implies that there is an $n$-vertex graph for which we can remove one edge from each $K_{4,4}$ and 
have a subgraph that is $K_{4,4}$-free and has at least $0.00004 n^{8/5}$ copies of $K_4$.  While simple, this 
argument does not improve (\ref{all 4 case}).


\section{Counting $r$-graphs of girth $5$ and the proof of Theorem \ref{counting theorem girth 5 intro}}\label{counting section}

For a family of forbidden subgraphs $\mathcal{F}$, denote by $F_r(n, \mathcal{F})$ the family of all $r$-uniform simple hypergraphs on $n$ vertices which do not contain any member of $\mathcal{F}$ as a subgraph and let $F_r(n, \mathcal{F}, m)$ denote those graphs in $F_r(n, \mathcal{F})$ which have $m$ edges. Let 
\begin{align*}
f_r(n, \mathcal{F}) &= |F_r(n, \mathcal{F})|
\\ f_r(n, \mathcal{F}, m) & = |F_r(n, \mathcal{F}, m)|.
\end{align*}
It is clear that 
\begin{equation}\label{trivial lower bound}
f_r(n, \mathcal{F}) \geq 2^{\mathrm{ex}_r(n, \mathcal{F})}.
\end{equation}
In this section, we will study the quantities $f_r(n, \mathcal{F})$ and $f_r(n, \mathcal{F},m )$ when $\mathcal{F}$ is the family of Berge cycles of length at most $4$. 
Let $\mathcal{B}_k = \{\textup{Berge-}C_2, \dots, \textup{Berge-}C_k\}$. 
Note that when a hypergraph is $\textup{Berge-}C_2$-free, this means that any two hyperedges share at most one vertex (i.e., the hypergraph is linear). 
Throughout this section, when we say a hypergraph of {\it girth} $g$, we mean an $r$-uniform hypergraph that is $\mathcal{B}_{g-1}$-free, i.e, it contains no Berge-$C_k$ for $k < g$.

Lazebnik and Verstra\"{e}te \cite{lv} examined girth $5$ hypergraphs and gave the following bounds for $r=3$
\[
\ex_3(n,\mathcal{B}_4) = \frac{1}{6}n^{3/2}+o(n^{3/2})
\]
and for general $r$ (with $n$ large enough),
\[
\frac{1}{4}r^{-4r/3}n^{4/3} \leq \ex_r(n,\mathcal{B}_4) \leq \frac{1}{r(r-1)}n^{3/2}+ O(n).
\]

Our main result in this section is the next theorem.  


\begin{theorem}\label{counting theorem edges girth 5}
	Let $r\geq 2$ and $n$ be large enough. Then
	\[
	f_r(n, \mathcal{B}_4, m) \leq \mathrm{exp}\left(n^{4/3}\log^3 n\right) \left(\frac{n^3}{m^2}\right)^m.
	\]
\end{theorem}

Theorem \ref{counting theorem edges girth 5} yields the following two corollaries, the first of which implies 
Theorem \ref{counting theorem girth 5 intro}.  

\begin{corollary}\label{counting theorem girth 5}
	Let $r\geq 2$. Then there exists a constant $C$ such that 
	\[
	f_r(n, \mathcal{B}_4) \leq 2^{C n^{3/2}}.
	\]
\end{corollary}

The first group to consider extremal problems in random graphs was probably Babai-Simonovits-Spencer~\cite{babsimspe}. Among others they asked: what is the maximum number of edges of a $C_4$-free subgraph of the random graph $G_{n,p}$ when $p=1/2$? Here we give a partial answer to the corresponding question in Berge-hypergraph setting. Let $G_{n,p}^{(r)}$ be the random $r$-uniform hypergraph on $n$ vertices, each edge being present independently with probability $p$.

\begin{corollary}\label{turan number in random graph}
	Let $0< p < \frac{1}{(r(r-1))^2}$. Then there exists an $\epsilon > 0$ such that with probability tending to $1$,
	\[
	\mathrm{ex}_r(G_{n,p}^{(r)}, \mathcal{B}_4) < (1-\epsilon)\mathrm{ex}_r(n, \mathcal{B}_4).
	\]
\end{corollary}

Theorem \ref{counting theorem edges girth 5} implies Corollary \ref{counting theorem girth 5} by noting that $(n^3/m^2)^m = 2^{O(n^{3/2})}$ and Corollary \ref{turan number in random graph} by a simple first moment argument combined with the fact \cite{lv} that $\mathrm{ex}_r(n, \mathcal{B}_4) \leq \frac{1+o(1)}{r(r-1)}n^{3/2}$.  

\begin{proof}[Proof of Theorem \ref{counting theorem edges girth 5}]
For a graph $H$ and a natural number $d$, let $\mathrm{ind}(H, d)$ denote the number of independent sets of size exactly $d$ in $H$.
We adapt the proofs of Kleitman's and Winston's upper bound on the number of $C_4$-free graphs \cite{KW} (see also \cite{S} for a nice exposition) and F\"uredi's extension to graphs with $m$ edges \cite{F}.  The rough idea of the proof is that any hypergraph of girth $5$ can be decomposed into a sequence of subhypergraphs satisfying mild conditions, and that the number of such sequences is bounded.
	
	If $G$ is any hypergraph, we may successively peel off vertices of minimum degree. Specifically, let $v_n$ be a vertex such that $d_G(v_n) = \delta(G)$. Once $v_n, v_{n-1}, \dots, v_{k+1}$ are chosen, let $v_k$ satisfy 
	\[
	|\Gamma(v_k)\setminus \{v_n,\dots, v_{k+1}\}| = \delta(G\setminus \{v_n, \dots, v_{k+1}\}).
	\]
	For each $i$, let $G_i = G[\{v_1,\dots, v_i\}]$. This sequence of subhypergraphs has the property that for all $i$,
	\[
	\delta(G_{i-1}) \geq \delta(G_i)-1 = d_{G_i}(v_i)-1.
	\]
	That is, $\delta(G_i) \leq \delta(G_{i-1}) + 1$. Now, if $G$ is $\mathcal{B}_4$-free, then each $G_i$ is also $\mathcal{B}_4$-free. To summarize, any hypergraph of girth $5$ may be constructed one vertex at a time such that 
	\begin{enumerate}
		\item At each step, the subhypergraph is $\mathcal{B}_4$-free.
		\item When adding the $i$'th vertex $v_i$, we have that the minimum degree of the graph which $v_i$ is being added to is at least $d_{G_i}(v_i) -1$.
	\end{enumerate}
	
	The crux of the upper bound is that one cannot add a vertex to a graph of high minimum degree and keep it $\mathcal{B}_4$-free in too many ways. To formalize this, let $g_i(d)$ be the maximum number of ways to attach a vertex of degree $d$ to a $\mathcal{B}_4$-free graph on $i$ vertices with minimum degree at least $d-1$, such that the resulting graph remains $\mathcal{B}_4$-free, and let $g_i = \max_{d\leq i} g_i(d)$. Note that 
	\begin{equation}\label{upper bound on g}
	g_i(d) \leq \binom{i}{(r-1)d}((r-1)d)!
	\end{equation}
	for all $d$, so $g_i$ is well-defined. Now let us count the number of sequences of subhypergraphs $G_1,\dots, G_n$ that can come from a hypergraph of girth $5$ with $m$ edges, $G$. Note that each $G$ of girth $5$ creates (once the vertices are ordered) a unique sequence $G_1,\dots, G_n$. First, we trivially bound the number of ways to order the vertices $(v_1,\dots, v_n)$ by $n!$, and we also trivially bound the number of degree sequences $\{d_{G_1}(v_1),\dots, d_{G_n}(v_n)\}$ by $n!$. By the way we have constructed the sequence $\{G_1,\dots, G_n\}$ and by the definition of $g_i(d)$, we have that 
	\[
	f_r(n, \mathcal{B}_4 , m) \leq n! n! \max \prod_{i=1}^n g_i(d_i),
	\]
	where the maximum is taken over all degree sequences such that $\sum d_i = m$.
	
	If $d_i \leq i^{1/3}\log i$, we use \eqref{upper bound on g} and have that, for large $i$,
	\[
	g_i(d_i) \leq i^{i^{1/3}\log^2 i}.
	\]
	
	From now on we will assume $d_i \geq i^{1/3}\log i$. Assume that $G_i$ is a hypergraph of girth $5$ on $i$ vertices with minimum degree at least $d$. We construct an auxiliary graph $H_i$ with vertex set $V(H_i) = V(G_i)$ and $xy\in E(H_i)$ if and only if there is a path of length $2$ from $x$ to $y$ in the hypergraph $G_i$. 
	
	Now we observe that in order to attach $v_{i+1}$ to $G_i$ and have the resulting graph $G_{i+1}$ remain $\mathcal{B}_4$-free, the neighborhood of $v_{i+1}$ must be an independent set in $H_i$. To see this, if $v_{i+1}\sim x$ and $v_{i+1}\sim y$ where $xy\in E(H_i)$, then there is a path of length $2$ in $G_i$ from $x$ to $y$. Now, if there exists a hyperedge $e\in E(G_{i+1})$ such that $\{x,y, v_{i+1}\} \subset e$, this creates a Berge-$C_3$ in $G_{i+1}$. Otherwise, the vertex $v_{i+1}$ creates a Berge-$C_4$ in $G_{i+1}$. 
	
	Therefore to bound $g_i(d_i)$ it suffices to give a uniform upper bound on $\mathrm{ind}(H_i, d_i)$.  To do this, we use a lemma of Kleitman and Winston, which is the original inspiration for the container method \cite{KW}.
	
	\begin{lemma}[Kleitman and Winston (cf \cite{KLRS, S}]\label{Kleitman and Winston lemma}
		Let $G$ be a graph on $n$ vertices. Let $\beta\in (0,1)$, $q$ an integer, and $R$ a real number satisfy
		\begin{enumerate}
			\item $R\geq e^{-\beta q} n$.
			\item For all subsets $U\subset V(G)$ with $|U| \geq R$,
			\[
			e_G(U) \geq \beta \binom{|U|}{2}.
			\]
		\end{enumerate}
		Then for all $m\geq q$,
		\[
		\mathrm{ind}(G, m) \leq \binom{n}{q}\binom{R}{m-q}.
		\]
	\end{lemma}
	
	We now give an upper bound on $\mathrm{ind}(H_i, d)$. Let $B\subset V(H_i)$. Then (with floors and ceilings omitted)
	\begin{align*}
	e_{H_i}(B) & \geq \sum_{z\in V(G_i)} \binom{ |\Gamma_{G_i}(z)\cap B|/(r-1)}{2} \\
	& \geq i  \binom{\frac{1}{(r-1)i} \sum_{z\in V(G_i)} |\Gamma_{G_i}(z) \cap B|}{2} \\ 
	& \geq i\binom{\frac{1}{(r-1)i} \sum_{y\in B} \frac{d(y)}{r}}{2} \\
	& \geq i \binom{\frac{|B| \delta(G_i)}{r^2 i}}{2} \geq i \binom{\frac{|B|(d_i-1)}{r^2 i }}{2}\\
	& \geq \frac{|B|^2d_i^2}{8r^4 i},
	\end{align*}
	where the last inequality holds for $i$ large enough. This quantity is bigger than 
	\[
	i^{-1/3}\log i\binom{|B|}{2}
	\]
	for $i$ large enough since $d_i \geq i^{1/3}\log i$. Now we let $\beta = i^{-1/3}\log i$ (which is in $(0,1)$ for $i$ large enough), $R = \frac{i}{d_i}$, and $q = i^{1/3}$. Note that $R > 1$ and $e^{-\beta q}i = 1$. Therefore by Lemma \ref{Kleitman and Winston lemma}, we have 
	\[
	\mathrm{ind}(H_i, d_i) \leq \binom{i}{i^{1/3}} \binom{\frac{i}{d_i}}{d_i - i^{1/3}}.
	\] 
	Since $d_i - i^{1/3} \geq \frac{1}{2} d_i$ for $i$ large enough, we have 
	\[
	\mathrm{ind}(H_i, d_i) \leq \left(\frac{2ei}{d_i^2}\right)^{d_i} (i^{2/3})^{i^{1/3}}.
	\]
	Thus
	\begin{align*}
	f_r(n, \mathcal{B}_4, m) &\leq n! n! \max \prod \left(\frac{2ei}{d_i^2}\right)^{d_i} (n^{2/3})^{2n^{1/3}\log^2n}
	\\ & \leq \mathrm{exp}\left( n^{4/3}\log^3 n + (\log n +O(1)) \sum d_i - 2\sum d_i \log d_i\right)
	\end{align*}
	for $n$ large enough. Next we note that $\sum d_i = m$ and by convexity $\sum d_i \log d_i \geq m \log(m/n)$. Rearranging gives the result.
\end{proof}


\end{document}